\documentclass[11pt]{amsproc}
\usepackage{mathrsfs}
\usepackage{stmaryrd}
\usepackage{cases}
\usepackage{amssymb}
\usepackage{amsmath}
\usepackage{amsfonts}
\usepackage{graphicx}
\usepackage{amsmath,amstext,amsbsy,amssymb,color}
\usepackage{multirow}

\newtheorem{theorem}{Theorem}[section]

\newtheorem{proposition}[theorem]{Proposition}

\theoremstyle{definition}

\theoremstyle{remark}

\numberwithin{equation}{section} \errorcontextlines=0

\renewcommand{\a}{\alpha}

\newcommand{\la}{\lambda}

\newcommand*{\longhookrightarrow}{\ensuremath{\lhook\joinrel\relbar\joinrel\rightarrow}}
\newcommand{\tr}{\mathrm{tr}}

\begin{document}

\title[Spin Characters of Hyperoctahedral Wreath Products]
{Spin Characters of Hyperoctahedral Wreath Products}
\author{Xiaoli Hu}
\address{Hu: School of Mathematics and Computer Science,
Jianghan University, Wuhan, Hubei 430056, China}
\email{xiaolihumath@163.com}
\author{Naihuan Jing$^*$}
\address{Jing: Department of Mathematics,
   North Carolina State University,
   Raleigh, NC 27695, USA
   }
\email{jing@math.ncsu.edu}
\keywords{hyperoctahedral groups; wreath products; spin
 characters}
\begin{abstract}
The irreducible spin character values of the wreath products of the hyperoctahedral groups with an arbitrary finite group are determined.
\end{abstract}
\thanks{Supported by NSFC grants 11271138, 11426116 and Simons Foundation grant 198129.}
\subjclass[2010]{Primary: 20C25; Secondary: 20C30, 20B30, 20C15} 
\thanks{$^\star$ Corresponding Author}
\maketitle

\section{ Introduction}

Irreducible spin characters of the symmetric group ${S}_n$ were studied by Schur \cite{Sch} as one
of the important examples in representation theory. Schur first derived a Frobenius-type
formula for the nontrivial spin
character values at conjugacy classes corresponding to
partitions with odd integer parts in terms of
Schur's $Q$-functions, which is similar to the case of the symmetric group. Furthermore, Schur showed that beyond the
Frobenius-type formula, there are other spin character values on special
conjugacy classes corresponding to (odd) strict partitions, which are not given by his Q-functions but
derived via his new twisted tensor product of basic spin representations.

After Schur's classic paper there have been quite a number of papers devoted to
the spin groups and in particular, hyperoctahedral spin groups and their generalizations. Morris gave
an alternative description of the twisted tensor product \cite{Mo} and derived an
iterative formula for Schur's $Q$-functions. Later Morris also studied double covering groups
of Weyl groups following Schur's theory \cite{Mo2}. Sergeev \cite{Se} showed that
representation theory of the twisted hyperoctahedral group $\widetilde{H}_n$ is similar to that of the spin group
$\widetilde{S}_n$ (cf. \cite{N, Jo1}),
and proved that Schur's $Q$-functions also served as generating functions for some
irreducible spin  supercharacters of the hyperoctahedral groups.
Finally Yamaguchi \cite{Ya1} clarified this relationship and established an equivalence between the twisted group
algebra of the symmetric group and the queer Lie superalgebra.

For the wreath product  $\widetilde{\Gamma}_n=\Gamma\wr \widetilde{S}_n$ of a finite group $\Gamma$ by the spin symmetric group $\widetilde{S}_n$, most part of the spin character table was determined by Frenkel,
Wang and the second author using the vertex operator calculus \cite{JW}. The remaining part was completed by the authors \cite{HJa, HJ} using
 Clifford algebras. In 2002, Wang and the second author \cite{JW} generalized the
 method to study the semi-direct product $\widetilde{H}\Gamma_n$ of the spin hyperoctahedral group $\widetilde{H}_n$ by
 a finite group $\Gamma$.  The status of the character theory can be easily described as follows. If we denote by
$\Gamma_*$ (resp. $\Gamma^*$) the set of conjugacy classes (resp. irreducible characters) of
$\Gamma$, then the irreducible spin characters of $\widetilde{H}\Gamma_n$ are determined by pairs of the strict partition-valued functions on $\Gamma^{*}$ and the spin character values at the even conjugacy classes are provided by vertex operators in \cite{JW}. However, the spin character values at odd strict partition-valued functions of $\Gamma_*$ are still unknown.

This paper aims to give the complete spin character table of the semi-direct product $\widetilde{H}\Gamma_n$,
which include projective characters of various Weyl groups and their wreath products. Part of the
motivation is to extend the previous type $A$ theory (the character theory of $\widetilde{\Gamma}_n$) to type $B$ or $D$. The following table shows the situation with the new case. Here $\mathcal{SP}^1_n(\Gamma_{*})$ (resp. $\mathcal{SP}^1_n(\Gamma^{*})$) denotes the set of odd strict partition-valued functions on $\Gamma_{*}$ (resp. $\Gamma^{*}$) and $\mathcal{OP}_n(\Gamma_{*})$ (resp. $\mathcal{OP}_n(\Gamma^{*})$) denotes the set of
odd integer partition-valued functions on $\Gamma_{*}$ (resp. $\Gamma^{*}$), where  $\Gamma_*$ (resp. $\Gamma^*$) denotes the set of
conjugacy classes (resp. irreducible characters) of the finite group $\Gamma$. The unknown part of the character values so far is
denoted by a question mark. Note that there are two parities for spin hyperoctahedral wreath products and parition-valued functions, it
happens that they are compatible for split conjugacy classes (see Prop. \ref{p:splitconj}).
\begin{table}[!hbp]
\begin{tabular}{|c|c|c|c|}
\hline
\multirow{2}{*}{}  &\multicolumn{2}{|c|}{$D_{\rho}^{\pm}$}\\
\cline{2-3}
 &&\\
Spin chars $\diagdown$ Conj. classes &$\rho\in\mathcal{OP}_n(\Gamma_{*})$ (even) &$\rho\in\mathcal{SP}^1_n(\Gamma_{*})$ (odd) \\
&&\\
\hline
$\nu\in \mathcal{SP}_n(\Gamma^{*}) $ & &\\
 $l(\nu)$ even&vertex~ operators&\quad  0 \quad\\
$\chi_{\nu}$&&\\
\cline{1-3}
$\nu\in \mathcal{SP}_n(\Gamma^{*}) $& & \\
 $l(\nu)$ odd&vertex~ operators&\quad  ? \quad\\
$\chi^{\pm}_{\nu}$&&\\
\hline
\end{tabular}
\caption{Spin character table of $\widetilde{H}\Gamma_n$}
\end{table}

The spin characters computable by vertex operators are actually character values that come from
certain supermodules of the wreath products of hyperoctahedral groups. It turns out that
the characters values at conjugacy classes indexed by partition-valued functions
with odd integer parts are the only non-zero values for certain modules, which then give the
first part shown in the table.

Our method to uncover the other spin character values is a detailed analysis
of difference characters for the ordinary modules and also how they are
related with the characters of the supermodules. We generalize
J\'ozefiak's supermodule approach \cite{Jo, Jo1} to study spin representations of $\widetilde{H}\Gamma_n$ and compute the associate spin characters at the odd strict conjugacy classes.
While we take the advantage of the similarity between the representation theories of
the wreath products and their hyperoctahedral counterparts, we also need to single
out the difference in order to determine the nontrivial spin character values of the
latter groups. Moreover, there is also another new point
for hyperoctahedral wreath products. Previous work has been emphasizing
 their representations as a super theory, however in the cur\-rent situation
 one needs to disengage the super theory to get their ordinary counterpart.

\section{The spin hyperoctahedral groups}

Let $S_n$ be the symmetric group on $n$ letters for $n\geq 4$. The hyperoctahedral group $H_n:=\mathbb{Z}_2\wr S_n$
 is  generated by transpositions $s_i=(i,i+1)\ (1\leq i \leq n-1)$
and involutions $b_j\  (1\leq j\leq n)$ with the relations
\begin{equation}
\begin{split}
s_i^2=1, (s_is_{i+1})^3=1, (s_is_j)^2=1  \mbox{~for~} |i-j|\geq 2,\\
s_ib_i=b_{i+1}s_i, s_ib_j=b_js_i \mbox{~for~} j\neq i, i+1,\\
b_j^2=1, b_ib_j=b_jb_i \mbox{~for~} i\neq j. 
\end{split}
\end{equation}
We denote by $\widetilde{H}_n$ the group generated by $s_i\ (1\leq i\leq n-1)$ and $a_j\ (1\leq j\leq n)$ and $z$
with the following relations:
 \begin{align}\label{eq,1}
 s_i^2=1, (s_is_{i+1})^3=1,\quad (s_is_j)^2=1 \mbox{~for~} |i-j|\geq 2, \\ \label{eq1-2}
 s_ia_i=a_{i+1}s_i,\quad s_ia_j= a_js_i \mbox{~for}~j\neq i,i+1,\\ \label{eq1-3}
 a_j^2=z, \quad a_ia_j=za_ja_i \mbox{~for~} i\neq j, 
 \end{align}
where $z$ belong to the center and $z^2=1$.

Note that
 $\widetilde{H}_n$ is a semidirect product of the symmetric group
 $S_n=\langle s_i| 1\leq i\leq n-1\rangle$ and the $2$-group generated by $a_j$ subject to
 the relations (\ref{eq1-3}).
Furthermore, $\widetilde{H}_n$ is a central extension of $H_n=\mathbb{Z}_2^n\rtimes S_n$ through
\begin{equation}\label{e:ext}
1\longrightarrow \langle z\rangle\longhookrightarrow
\widetilde{H}_n \xrightarrow{\text{$\ \ \theta_n \ \ $}} H_n\longrightarrow 1,
\end{equation}
where $\theta_n$ is the group homomorphism sending  $s_i$ to $(i,i+1)$ for $1\leq i\leq n-1$, $a_j$ to $b_j$ for $1\leq i\leq n$, and $z$ to 1.
The group $\widetilde{H}_n$ is $\mathbb Z_2$-graded with the
parity defined by $p(a_i)=1$ and $p(s_i)=p(z)=0$. With this parity the group algebra $\mathbb{C}(\widetilde{H}_n)$ becomes a superalgebra.

The spin symmetric group  $\widetilde{S}_n$ is  the finite group
generated by $t_i \ (1\leq i\leq n-1)$ and the central element $z$ with the relations
\begin{equation}
 t_i^2=z,\quad z^2=1,\quad(t_it_{i+1})^3=1,\quad t_it_j=zt_jt_i \mbox{~for~} |i-j|\geq 2.
\end{equation}
 For each $n>1$, let $\mathscr{A}_n$ denote the algebra generated by $\varsigma_i \ (1\leq i \leq n-1)$ and satisfy the following relations:
\begin{equation}\label{e:algA}
\begin{split}\varsigma_i^2&=-1 \ \ \ (1\leq i\leq n-1),\\
 (\varsigma_i\varsigma_j)^2&=-1\ \ \  (|i-j|\geq 2),
\\
(\varsigma_i\varsigma_{i+1})^3&=-1\ \ \ ( 1\leq i\leq n-2).
\end{split}
\end{equation}
The algebra $\mathscr{A}_n$ is also a superalgebra under the parity
that $d(\varsigma_i)=1\ (1\leq i\leq n-1)$. In fact
$\mathscr{A}_n$ is isomorphic (as a superalgebra) to the twisted group algebra
$\mathbb{C}(\widetilde{S}_n)/(1+z)$ via
$\varsigma_i\mapsto t_i$ \cite{Sch}.

For $n>1$, let ${\mathscr{B}}_n$ be the algebra generated by
$\omega_i\ (0\leq i \leq n-1)$
satisfying
\begin{align}
& \omega^{2}_i=1 ~(0\leq i\leq n-1), \ (\omega_i\omega_{i+1})^3=1 ~(1\leq i \leq n-2), \\
&(\omega_i\omega_j)^2=1 ~ (|i-j|\geq 2), \  \  (\omega_0\omega_1)^4=-1.
\end{align}
The algebra ${\mathscr{B}}_n$ is the group algebra of a double covering group of the Coxeter group $H_n$
and also a superalgebra by defining $d(\omega_0)=1$ and $d(\omega_i)=0$ for $1\leq i \leq n-1$. Moreover ${\mathscr{B}}_n\cong \mathbb{C}[\widetilde{H}_n]/(1+z)$ via $\omega_i\mapsto t_i$ ($1\leq i\leq n-1$) and $\tau_i=\omega_{i-1}\cdots \omega_1\omega_0\omega_1\cdots\omega_{i-1}\mapsto a_{i}$ for $i=1, \ldots, n$.
In particular, $\omega_0\equiv \tau_1$.

Let ${\mathscr{C}}_n$ be the Clifford algebra generated by $\xi_i\ (i=1,\ldots,n)$ subject to the relations
 $$\xi_i^2=1, \ \ \xi_i\xi_j=-\xi_j\xi_i\ \ (i\neq j).$$
It becomes a superalgebra under the parity $d(\xi_i)=1$. 
Let ${\mathscr{C}}_n\hat{\otimes} {\mathscr{A}}_n$ denote the super tensor product of the superalgebras ${\mathscr{C}}_n$ and ${\mathscr{A}}_n$. It is proved in \cite{Ya1} that the super tensor product
${\mathscr{C}}_n\hat{\otimes} {\mathscr{A}}_n\simeq {\mathscr{B}}_n$ as superalgebras under the map $\vartheta$:
\begin{equation}\label{e:theta}
\begin{split}
\xi_i\otimes 1&\mapsto \tau_i \ \ (1\leq i\leq n), \\
1\otimes \varsigma_j&\mapsto \frac{1}{\sqrt{2}}(\tau_j-\tau_{j+1})\omega_j, \ \ (1\leq j\leq n-1)
\end{split}
\end{equation}
where $\tau_i$ are defined above.

\section{Conjugacy classes of wreath products}

 For a finite  group $\Gamma$, set
$\Gamma^n=\Gamma\times\cdots\times\Gamma$, the $n$-th direct product of $\Gamma$. The twisted
hyperoctahedral group $\widetilde{H}_n$  acts on
$\Gamma^n$ by permuting the factors:
\begin{equation}
\begin{split}
z\cdot(g_1,\ldots,g_n)=&(g_1,\ldots,g_n),\\
a_i\cdot(g_1,\ldots,g_n)=&(g_1,\ldots,g_n),\ \ i=1,\ldots,n\\
\sigma\cdot(g_1,\ldots,g_n)=&(g_{\sigma^{-1}(1)},\ldots,g_{\sigma^{-1}(n)}),
\sigma\in S_n.
\end{split}
\end{equation}
Let $\widetilde{H}\Gamma_n$ be
the semi-direct product of $\Gamma^n$ by  $\widetilde{H}_n$  with
the multiplication:
\begin{equation}
(g,a\sigma)\cdot(g^{'},a^{'}\sigma^{'})=(g\sigma(g^{'}),
a\sigma(a^{'})\sigma\sigma^{'}),
\end{equation}
where $g,g^{'}\in \Gamma^n, a, a^{'}\in \Pi_n,
\sigma,\sigma^{'}\in S_n$. From now on, we write $H\Gamma_n$ for the wreath product of $S_n$ by $ \Gamma\times\mathbb{Z}_2 $. Note that $\widetilde{H}\Gamma_n$ is a
double cover of  $H\Gamma_n$.
In the following we use the notations of \cite{JW} for $\widetilde{H}\Gamma_n$.
For a subset $I=\{i_1,\ldots,i_m\}$ of
$\{1,\ldots,n\}$, we set $a_{I}=a_{i_1}a_{i_2}\cdots a_{i_m} \in\widetilde{H}_n$ and $b_I=b_{i_1}\cdots b_{i_m}\in H_n$, where $b_{i}\ (i=1,\ldots,n)$ are the
generators of $\mathbb{Z}_2^n$. 

We recall that a partition $\la$ of $n$ is a sequence  of non increasing positive integers $\la_i$ called parts
such that $\sum_i\la_i=n$. We denote by $l(\lambda)$ the number of parts or length of $\lambda$, and $\parallel\lambda\parallel=\la_1+\la_2+\cdots$ the weight of $\lambda$. A partition $\la$ is called a {\it strict partition} if all parts $\la_i$ are distinct.
We also use the notation $\la=(1^{m_1(\la)}2^{m_2(\la)}\cdots)$ for the partition $\la$, where
$m_k(\la)$ is the multiplicity of $k$ in the parts $\la_i$'s.

The conjugacy classes of ${H}\Gamma_n$ are parametrized by
 a pair of partition-valued functions.
For a cycle $\sigma=(i_1\ldots i_m)\in S_n$, the $supp(\sigma)$ is defined to be $I=\{i_1,\ldots,i_m\}$. For $(g,
b_{I}\sigma)\in H\Gamma_n,$ the element $b_I\sigma$ can be
uniquely written (up to order) as a product
$$b_I\sigma=(b_{I_1}\sigma_1)\cdots(b_{I_m}\sigma_m),$$
where $\sigma=\sigma_1\cdots\sigma_m$ is a product of disjoint
cycles of $S_n$ and $I_k\subset supp(\sigma_k)$ for $k=1,\ldots, m$ \cite{JW}.
We call $b_{I_a}\sigma_a$ a signed cycle of $b_I\sigma$
with the sign $(-1)^{|I_a|}$. For each signed cycle
$b_{I_k}\sigma_k$ with $\sigma_k=(j_1,\ldots,j_m)$
the {\it {signed cycle-product}} of $b_{I_k}\sigma_k$ is defined to be
$$(-1)^{|I_k|}g_{j_m}g_{j_{m-1}}\cdots g_{j_2}g_{j_1}.$$

 For each $c\in \Gamma_*$, let
$m^{+}_k(c)$ (resp. $ m_k^{-}(c))~(k\geq 0)$ be the number of
$k$-cycles of the permutation $\sigma$ such that its positive (resp. negative) signed cycle-product
lies in the conjugacy class $c$. For each $c\in \Gamma_*$, set $$\rho^{+}(c)=(1^{m_1^{+}(c)}
2^{m_2^{+}(c)}\ldots ) \ \ \hbox{and}\ \ \rho^{-}(c)=(1^{m_1^{-}(c)}
2^{m_2^{-}(c)}\ldots ),$$
which are partitions indexed by $c\in\Gamma_{*}$. Then $$\rho^{+}=(\rho^{+}(c))_{c\in\Gamma_{*}}, \ \ \rho^{-}=(\rho^{-}(c))_{c\in\Gamma_{*}}$$
are two signed partition-valued functions on $\Gamma_{*}$.
Moreover, $\rho=(\rho^{+},\rho^{-})$ defines a pair of partition-valued functions on $\Gamma_*$ with $l(\rho)=l(\rho^{+})+l(\rho^{-})$,
where $l(\sigma)$ is the total length of $\sigma$, and
$||\rho^{+}||+||\rho^{-}||=||\rho||$, the total weight of the partition-valued function.
It is proved \cite{JW} that the conjugacy classes of $H\Gamma_n$ are
indexed by pairs of partition-valued functions $\rho=(\rho^{+},\rho^{-})$ such that
$||\rho||=n$.
In particular,
$\rho=(\rho^{+},\rho^{-})$ is called the type of the element
$(g,b_{I}\sigma)\in H{\Gamma}_{n}$. Two elements of $H\Gamma_n$ are conjugate if and only if they have the same type.

We write $C_{\rho^{+},\rho^{-}}$ for the conjugacy
class corresponding to  elements of type $(\rho^{+},\rho^{-})$. Furthermore, as any element of $H\Gamma_n$ is of the form $(g,b_I\sigma)$, so
a conjugacy class $C_{\rho^{+},\rho^{-}}$ is called {\it{even}} or {\it{odd}} if the cardinality $|I|$ is even or odd.
Thus the general element of $\widetilde{H}\Gamma_n$ is of the
form $(g,(-1)^pa_I\sigma)$ for $p=\pm1$, where
$(-1)^pa_I\sigma=(-1)^p(a_{I_1}\sigma_1)\cdots (a_{I_q}\sigma_q)$ and
$\sigma=\sigma_1\cdots\sigma_q$ is a disjoint union of cycles in
$ S_n$ \cite{JW}. We define the {\it parity} $p$ on $\widetilde{H}\Gamma_n$ by
$p(g,(-1)^pa_I\sigma)=|I| \mod 2$. Note that
this parity generalizes the parity $p$ for $\widetilde{H}_n$ (below \eqref{e:ext}). Therefore the group algebra $\mathbb C[\widetilde{H}\Gamma_n]$
becomes a superalgebra.

Let $\mathcal{P}_n(\Gamma_{*})$ denote the set of partition-valued functions of $n$ on $\Gamma_*$. Let $\mathcal{OP}_n(\Gamma_{*})$ (resp. $\mathcal{SP}_n(\Gamma_{*})$) be the subset of $\mathcal{P}_n(\Gamma_{*})$ such that each part is odd (resp. every partition is strict).
For $i=0, 1$, let $\mathcal{SP}^{i}_n(\Gamma_{*})=\{\rho\in \mathcal{SP}_n(\Gamma_{*})| i=(n-l(\rho))\mod 2\}$. When $\Gamma=1$, the partition-valued functions become partitions, so we denote $\mathcal{OP}_n$, $\mathcal{SP}_n$  and $\mathcal{SP}^{i}_n (i=0, 1)$.
It is well-known that the set $\mathcal{SP}_n$ of strict partitions of $n$ is
one-to-one correspondent to the set $\mathcal{OP}_n$ of partitions of $n$ by odd integers \cite{An}. 

Denote by $D_{\rho^{+},\rho^{-}}$ the inverse image in $\widetilde{H}\Gamma_n$ of $H\Gamma_n$-conjugacy class $C_{\rho^{+},\rho^{-}}$ under the homomorphism $\theta_n$. For any
$x\in C_{\rho^{+},\rho^{-}}$, if $x$ is conjugate to $zx$ then $D_{\rho^{+},\rho^{-}}$ is also a $\widetilde{H}\Gamma_n$-conjugacy class; Otherwise,
$D_{\rho^{+},\rho^{-}}$ is a split class, in this case $C_{\rho^{+},\rho^{-}}$ (resp.$D_{\rho^{+},\rho^{-}}$) is called a split conjugacy class in $H\Gamma_n$ (resp.$\widetilde{H}\Gamma_n$). It is easy to see that only split conjugacy classes of $\widetilde{H}\Gamma_n$ can support nonzero spin character values.

\begin{proposition} \cite{JW}\label{p:splitconj}
The conjugacy class $C_{\rho^{+},\rho^{-}}$ in $H\Gamma_n$ is split into two conjugacy classes if
and only if

 (1) when  $C_{\rho^{+},\rho^{-}}$ is even, we have
$\rho^{+}\in \mathcal
 {OP}_n(\Gamma_{*})$ and $\rho^{-}=\emptyset.$

(2) when $C_{\rho^{+},\rho^{-}}$ is odd, we have $\rho^{+}=\emptyset$ and $\rho^{-}\in \mathcal
 {SP}_n^1(\Gamma_{*})$.
\end{proposition}

 In view of Prop. \ref{p:splitconj}, when $C_{\rho^{+},\rho^{-}}$ is split,  we can simply denote $D_{\rho}=\theta_n^{-1}(C_{\rho^{+},\rho^{-}})=D_{\rho}^{+}\cup D_{\rho}^{-}$,
where $D_{\rho}^{-}=z D_{\rho}^{+}$. Consequently
the order of the centralizer of an element in $D_{\rho}$ of $\widetilde{H}\Gamma_n$ is given by
\begin{equation*}
\widetilde{Z}_{\rho}=2^{1+l(\rho)}Z_{\rho},
 \end{equation*}
where $Z_{\rho}=\prod_{c\in\Gamma_*, i\in\mathbb N}i^{m_i(c)}m_i(c)!\zeta_c^{l(\rho(c))}$  and $\zeta_c$ is the order of the centralizer of an element in
conjugacy class $c$ of $\Gamma$. For convenience, we also denote
$z_{\rho(c)}=\prod_{i\in \mathbb N}i^{m_i(c)}m_i(c)!$, so $Z_{\rho}=\prod_{c\in\Gamma_*}z_{\rho(c)}$.

\section{Character formulas for $\mathscr{A}_n$ and $\mathscr{B}_n$}
A {\it spin representation} of $\widetilde{S}_n (or\ \widetilde{H}_n)$ is an ordinary representation $\pi$ of $\widetilde{S}_n (or\ \widetilde{H}_n)$
 such that $\pi(z)=-1$.
From superalgebra viewpoint J\'{o}zefiak \cite{Jo1} computed the spin super-character values of $\widetilde{H}_n$ based on
\cite{Sch}. A general discussion was given earlier in Morris and Jones \cite{MJ} on spin characters. The foundation of all these works
is Schur's theory \cite{Sch} that most character values of the
spin symmetric group $\widetilde{S}_n$ are given by transition matrices
between the basis of Schur's $Q$-functions $Q_{\la}$
and that of power-sum symmetric functions $p_{\mu}$ in the subring
of symmetric functions generated by the power-sums of odd degrees: $p_1, p_3, \ldots$.
For basic information on Schur's $Q$-functions, see Macdonald's classic monograph \cite{Ma}.
As the exceptional spin character values are not given by symmetric functions,
we will compute them for $\widetilde{H}_n$ in this section.

Recall that $\mathscr A_n$ a quotient of the group algebra $\mathbb C\widetilde{S}_n$ (see \eqref{e:algA}). As a superalgebra $\mathscr{A}_n$ is a direct product of finitely many simple superalgebras.
For background materials on superalgebras, we refer the reader to \cite{Jo}. For a modern account of Schur's twisted tensor products, see also \cite{St}.
Simple (complex) superalgebras have two types: (1) Type $M$,
$M(r|s)$ is the block $2\times 2$ matrices with the main diagonal blocks
are $r\times r$-matrices and $s\times s$ matrices. The degree zero part consists of diagonal blocks
and degree one is formed by off-diagonal blocks.
(2) Type $Q$, $Q(n)$ is a block $2\times 2$ matrices
whose two main
diagonals and off-diagonal are the same $n\times n$ matrices. The $\mathbb Z_2$-grading is similar to
that of type $M$.
We will call an irreducible supermodule type $M$ (resp. $Q$) if it corresponds to the simple
superalgebra of type $M$ (resp. $Q$).

The Clifford superalgebra $\mathscr C_n$
is a simple superalgebra of type $M$ (resp. $Q$) according to $n$ even (resp. odd).
Subsequently $L_n=\mathscr C_n$ is a simple $\mathscr C_n$-module of type $M$ (resp. $Q$) if $n$ is even (resp. odd).

\begin{proposition}\label{p:Sch}(Schur \cite{Sch})
(1) The irreducible $\mathscr{A}_n$-supercharacters are indexed by $\mathcal {SP}_n$.
 For $\nu\in \mathcal {SP}_n$, let $\zeta_{\nu}$ be the character with type $\nu$, then it is completely determined by its values $\zeta_{\nu}^{\a}$ at the conjugacy class indexed by $\a \in \mathcal {OP}_n$ and
\begin{equation}
    Q_{\nu}=\sum_{\la\in \mathcal{OP}_n} 2^{\frac{l(\nu)+l(\la)+p(\nu)}{2}}z^{-1}_{\la}\zeta^{\la}_{\nu}p_{\la},
\end{equation}
where 
$p(\nu)$ is the parity of $\nu$ and $z_{\lambda}$ is the order of the centralizer of an element with cycle type $\lambda$ in $S_n$.

(2) The supercharacter $\zeta_{\nu}$ is of type $M$ (resp. $Q$) when
$n-l(\nu)$ is even (resp. odd).
\end{proposition}

For $\nu\in\mathcal{SP}_n$, let $V_{\nu}$ be the $\mathscr{A}_n$-module associated with $\nu$. Since $\mathscr{B}_n\simeq \mathscr{C}_n\hat{\otimes} \mathscr{A}_n$ and $\mathscr C_n$ has only one simple supermodule $L_n$,
 any irreducible super modules of $\mathscr{B}_n$ is a super tensor product $L_n\hat\otimes V_{\nu}$, where $V_{\mu}$ is
 irreducible $\mathscr{A}_n$-supermodule. Subsequently there is a one-to-one correspondence between irreducible $\mathscr{A}_n$-modules and
irreducible $\mathscr{B}_n$-modules.

\begin{proposition}\label{JO HG}(J\'{o}zefiak \cite{Jo2})
(1) The irreducible supercharacters of $\mathscr{B}_n$ are completely indexed by the strict partitions of $n$. For $\nu\in \mathcal{SP}_n$,  if $l(\nu)$ is even (resp. odd) then the corresponding supercharacter $\xi_{\nu}$ is of type $M$ (resp. $Q$).
\\
(2) For each $\nu\in \mathcal{SP}_n$, the supercharacter values are determined by
\begin{equation}
    Q_{\nu}=\sum_{\la\in \mathcal{OP}_n}2^{[\frac{l(\nu)}{2}]} z^{-1}_{\la}\xi^{\la}_{\nu}p_{\la}.
\end{equation}
where $[x]$ denotes the largest integer $\leq x$.
\end{proposition}

Now we consider the relationship between characters of modules and super modules.

A spin representation $\pi$ of $\widetilde{S}_n$ (resp. $\widetilde{H}_n$) is called of {\it double spin}
or {\it self-associate spin} if $\pi\cong (-1)^{sgn}\circ\pi$. Otherwise, $\pi$ and $\pi^{'}:=(-1)^{sgn}\circ\pi$ are called a {\it pair of associate spin} representations. Here $(-1)^{sgn}$ is the one-dimensional sign representation of $\widetilde{S}_n$ (resp. $\widetilde{H}_n$).
The following correspondence between ordinary spin characters
 and super spin characters is crucial for our discussion.

First of all, if $V$ is a super module of the superalgebra $\mathscr A_n$ (resp. $\mathscr B_n$), then
$\chi_{V}(x)=trace_{V}(x)=0$ for any element $x$ of degree one.
On one hand, the character $\pi$ of a spin super module
 of type $M$ is also a double spin character by forgetting
 the super structure. If $\pi$ is the character of an irreducible supermodule of type $Q$, then $\pi=\chi^++\chi^-$,
 where $\chi^{\pm}$ are a pair of irreducible associate spin characters corresponding to $\pi$. Here $\chi^{-}$($x$)=$(-1)^{p(x)}\chi^{+}$($x$), where $p$ is the parity of the superalgebra. For convenience, we denote the difference
 $\nabla(\chi)=\chi^{+}-\chi^-$, which can be used to recover the associate spin characters
 from that of the supermodule.
 In summary, for $\widetilde{S}_n$ (or $\widetilde{H}_n$) we can get the ordinary
 spin character values on even conjugacy classes from the supercharacter formulas of $\mathscr{A}_n$ (or $\mathscr{B}_n$). Moreover, a double spin character is fully determined
 by its values on even conjugacy classes or degree 0 elements
 (as it vanishes on degree 1 elements),
 while associate spin characters may have non-zero values on odd conjugacy
 classes or degree 1 elements.

 The following well-known result shows how associate spin characters
 take values at the odd conjugacy classes.

\begin{proposition}(Schur \cite{Sch})\label{Sch} For $\nu \in\mathcal {SP}_n^1$ , there corresponds a pair of associate spin characters $\zeta^{+}_{\nu}$ and $\zeta^{-}_{\nu}$ of $\mathscr{A}_n$, and
$$\zeta^{+}_{\nu}(\nu^{+})=(\sqrt{-1})^{\frac{n-l(\nu)+1}{2}}\sqrt{\frac{\nu_1\nu_2\cdots\nu_l}{2}}, \ \
 \zeta^{+}_{\nu}(\mu^{+})=0  ~for~ \mu\neq\nu\in  \mathcal {SP}_n^1.$$
Moreover $\zeta^{-}_{\nu}(\mu^{+})=-\zeta^{+}_{\nu}(\mu^{+})$, where $\mu^{+}$ is an element in $D^{+}_{\mu}$.
\end{proposition}

 On the other hand, if $V$ is a double spin irreducible $\mathscr{B}_n$-module,
then $V$ can be naturally equipped with a $\mathscr{B}_n$-supermodule of type M as follows. Suppose $V\simeq V'=(-1)^{sgn}\otimes V$
is irreducible, where $sgn$ is the parity $p$ defined below \eqref{e:ext}.
As a vector space $V'=V$, the isomorphism map $H\in \mathrm{GL}(V)$ and satisfies that
\begin{equation*}
Hg=(-1)^{sg(g)}gH, \qquad \forall g\in \mathscr{B}_n.
\end{equation*}
Therefore $H^2$ commutes with $\mathscr{B}_n$, so it is a scalar by Schur's lemma. We can assume that $H^2=1$.
Let $V_0$ be the $1$-eigenspace of $H$ and $V_1$ the $(-1)$-eigenspace of $H$, then $V=V_0\oplus V_1$ gives rise to the desired $\mathscr{B}_n$-supermodule.
Moreover, both $V_i$ are $\mathscr{B}_n^{(0)}$-modules affording the character $\chi_{V_i}$, then for $x\in
\mathscr{B}_n^{(0)}$ one has that
\begin{align}
\tr_{V}(x)&=\chi_V(x):=\chi_{V_0}(x)+\chi_{V_1}(x)\\
\tr_V(Hx)&=\delta_V(x):=\chi_{V_0}(x)-\chi_{V_1}(x).
\end{align}
Here the latter is usually called {\it the difference character} of $V$ \cite{Sch}.
We will extend this notation to the spin character of a super module of type $M$, when
viewed as a character for the subalgebra of degree zero. Note that $\chi(x)=0$ if $x\in \mathscr{B}_n^{(1)}$
in this case.

Also, if $V$ is an irreducible associate spin $\mathscr{B}_n$-module, then
$D(V)=V\oplus V'$ becomes an irreducible $\mathscr{B}_n$-supermodule where
$D(V)_0=\{(v, v)|v\in V\}$, $D(V)_1=\{(v, -v)|v\in V\}$ and
the action is induced from that of the ordinary module, i.e.
$g^{(i)}(u, v)=(g^{(i)}u, (-1)^ig^{(i)}v)$ for $g^{(i)}\in\mathscr{B}_n^{(i)}$, the degree
$i$-subspace of $\mathscr{B}_n^{(i)}$ (with respect to the parity $p$).
Moreover, any irreducible associate spin module can be realized this way.

Conversely, any irreducible supermodule $U=U_0\oplus U_1$ of type $Q$ is of the form $D(V)$.
In fact, let $(-1)^p: (x, y)\mapsto (x, -y)$ be the parity endomorphism
of $U$. Then $V=\{\frac12{(v+(-1)^pv)}| u\in U\}$ and $V'=\{\frac12{(v-(-1)^pv)}| u\in U\}$.

From now on till (\ref{eqf3}), we assume that $\{\mathscr B_n\}$ is a tower of
finite-dimen\-sional superalgebras:
\begin{align*}
\mathscr B=\bigoplus_{n=0}^{\infty}\mathscr B_n,
\end{align*}
where
\begin{align*}
\mathbb C=\mathscr B_0\hookrightarrow\mathscr B_1\hookrightarrow\cdots\hookrightarrow\mathscr B_n\hookrightarrow\cdots
\end{align*}
and $\mathscr B$ a Hopf algebra under the natural multiplication and comultiplication:
\begin{align*}
m: &\mathscr B_m\hat\otimes\mathscr B_n\overset{\mathrm{Ind}}{\longrightarrow}\mathscr B_{m+n},\\
\Delta: &\mathscr B_m\longrightarrow \oplus_{i=0}^m \mathrm{Res} \mathscr{B}_i\hat\otimes \mathrm{Res}\mathscr{B}_{m-i}.
\end{align*}
Two nontrivial examples of the tower system $\{\mathscr B_n\}$ for superalgebras
were studied first in \cite{FJW, JW}. See also \cite{B} for the case of $\widetilde{S}_n$. 

For two spin supermodules $U$ and $V$ of $\mathscr{B}_{m}$ and $\mathscr{B}_n$, we define the
super (outer)-tensor product $U{\hat{\otimes}} V$ as a $\mathscr{B}_m\hat{\otimes}\mathscr{B}_n$-supermodule by
$$(x, y)(u\otimes
 v)=(-1)^{p(x)p(y)}(xu\otimes yv),$$
where $x\in\mathscr{B}_m$ and $y\in\mathscr{B}_n$ are homogeneous elements. In particular, $-1$ of $\mathscr{B}_{m+n}$
is identified with $(-1, 1)$ or $(1, -1)$ inside $\mathscr{B}_m\hat{\otimes}\mathscr{B}_n\hookrightarrow\mathscr{B}_{m+n}$.
Then $U\hat{\otimes}V$ is a spin
$\mathscr{B}_m\hat{\otimes}\mathscr{B}_n$-supermodule.
Moreover, let $U$ and $V$ be irreducible supermodules for $\mathscr{B}_m\hat{\otimes}\mathscr{B}_n$ respectively, then we have that \cite{Jo}

 (1) if both $U$ and $V$ are of type $M$, then $U\hat{\otimes} V$ is a simple
 $\mathscr{B}_m\hat{\otimes}\mathscr{B}_n$-supermodule of type $M;$

 (2) if  $U$ and $V$ are of different types, then $U\hat{\otimes }V$ is a simple
  $\mathscr{B}_m\hat{\otimes}\mathscr{B}_n$-supermodule of type $Q;$

 (3) if both $U$ and $V$ are of type $Q$,
  then $U\hat{\otimes} V\simeq N\oplus \overline N$ for some simple $\mathscr{B}_m\hat{\otimes}\mathscr{B}_n$-supermodules $N$ and $\overline N$ of type $M.$

 The simple summands of $U\hat{\otimes} V$ in (3) are constructed as follows. Let
 $H_1$ (resp. $H_2$) be the endomorphism $(-1)^p$ for $U$ (resp. $V$). Then
 \begin{align*}
  N_0&=\{u\otimes v+H_1u\otimes H_2v+i(u\otimes H_2v+H_1u\otimes v)| u\in U, v\in V\},\\
  N_1&=\{u\otimes v-H_1u\otimes H_2v+i(u\otimes H_2v-H_1u\otimes v)| u\in U, v\in V\},\\
  \overline N_0&=\{u\otimes v+H_1u\otimes H_2v-i(u\otimes H_2v+H_1u\otimes v)| u\in U, v\in V\},\\
  \overline N_1&=\{u\otimes v-H_1u\otimes H_2v-i(u\otimes H_2v-H_1u\otimes v)| u\in U, v\in V\}.
 \end{align*}
Let $-$ be the map $N\rightarrow \overline N$ taking
$i$ to $-i$ while keeping other part intact (like conjugation).
Then $\alpha(x): x\mapsto (-1)^{p(x)}\overline{x}$ establishes
a degree 1 isomorphism from $N$ to $\overline N$ as $\mathscr{B}_m\hat{\otimes}\mathscr{B}_n$-supermodules.

Using the method of \cite{J} it can be verified that this super tensor product satisfies the associativity.
Then we can pass this to usual spin modules as follows.
Let $V$ be an irreducible spin module. If $V$ is an irreducible double spin module, then $V$ can be naturally
given an irreducible supermodule structure of type M as described above
(still denoted by the same symbol for the supermodule). If $V$ is an irreducible associate spin module, then $D(V)=V\oplus V'$ is an irreducible supermodule
of type $Q$, and $D(V)_0\simeq V$ as ordinary modules. The following useful result can be proved
by a similar method to
\cite{Jo}, and we have added new values for later consideration of multi-products.

\begin{proposition}\label{tensorprod} Let $\{\mathscr B_n\}$ be a tower system of finite dimensional superalgebras.
Let $f$ and $g$ be the spin characters afforded by an irreducible
$\mathscr{B}_m$-module $U$ and an irreducible $\mathscr{B}_n$-module $V$ respectively.
Let $x=x_0+x_1\in\mathscr{B}_m$, $y=y_0+y_1\in\mathscr{B}_n$.

(i) If both $U$ and $V$ are double spin, then the super tensor product $U\hat{\otimes}V$ is
irreducible both as a supermodule of type M and an ordinary double spin
module for $\mathscr{B}_m\hat{\otimes}\mathscr{B}_n$. In this case,
$f(x_1)=g(y_1)=f\circledast g(x_i, y_{1-i})=0$ and
\begin{align*}
f\circledast g(x_0, y_0)&=f(x_0)g(y_0),\\
\delta(f\circledast g)(x_i, y_i)&=\delta(f)(x_i)\delta(g)(y_i).
\end{align*}

(ii) If $U$ is double spin and $V$ is associate spin,
then the super tensor product $U\hat{\otimes}D(V)$ is irreducible as a $\mathscr{B}_m\hat{\otimes}\mathscr{B}_n$-supermodule of type Q and decomposes
into $U\circledast V\oplus (U\circledast V)'$ as an ordinary module, where $(U\circledast V)'$
is the associated module of the irreducible module $U\circledast V$.
In this case, $f(x_1)=f\circledast g(x_1, y_0)=0$ and the following identities hold.
\begin{align*}
f\circledast g(x_0, y_0)&=f(x_0)g(y_0),\\
f\circledast g(x_0, y_1)&=\delta(f)(x_0)g(y_1),\\
\Delta(f\circledast g)(x_0, y_0)&=f(x_0)\Delta(g)(y_0).
\end{align*}

(iii) If both $U$ and $V$ are associate spin, so they give rise irreducible super spin modules
$D(U)$ and $D(V)$ of type $Q$. The super tensor product $D(U)\hat{\otimes}D(V)$ decomposes
into $W\oplus W$, where $W$ is an irreducible $\mathscr{B}_m\hat{\otimes}\mathscr{B}_m$-supermodule
of type M. Denote by $W=U\circledast V$ when it is viewed as an ordinary irreducible module
(up to isomorphism). In this case, we have that
\begin{align*}
f\circledast g(x_1, y_0)&=f\circledast g(x_0, y_1)=\delta(f\circledast g)(x_0, y_0)=0\\
f\circledast g(x_0, y_0)&=2f(x_0)g(y_0),\\   
\delta(f\circledast g)(x_1, y_1)&=2\sqrt{-1}f(x_1)g(y_1).
\end{align*}
\end{proposition}

Let $\mu$ be a partition,
and suppose $V_i$ $(1\leq i\leq m$) is
an irreducible spin (ordinary) $\mathscr{B}_{\mu_i}$-module.
Since the starred tensor product $\circledast$ is associative \cite{Jo, J} (see our notation above),
$V_1\circledast\cdots\circledast V_s$ is a well-defined spin $\hat\otimes\mathscr{B}_{\mu_i}$-module. Let $f_i$ be the character of $V_i$, and assume that the first $r$ of them are irreducible double spin (corresponding to type $M$) and the latter $k=m-r$ are irreducible associate spin (corresponding to type $Q$).
Then $f_1\circledast\cdots \circledast f_m$ is the
character of an irreducible summand of the super tensor product
\[
V_1\hat{\otimes}\cdots \hat{\otimes}V_r\hat{\otimes} D(V_{r+1})\hat{\otimes}\cdots \hat{\otimes}D(V_{r+k}).
\]
Note that $V_1\circledast\cdots \circledast V_{m}$ is only defined up to isomorphism. Repeatedly using
Prop. \ref{tensorprod}, we obtain that
\begin{align}\label{eqf}
f_1&\circledast\cdots\circledast f_{m}({x}_1,\ldots, {x}_{m})
=2^{[\frac{k}2]}f_1({x}_1)\cdots f_k({x}_m), \\ \nonumber
f_1&\circledast\cdots\circledast f_{m}({x}_1,\ldots {x}_{r}, y_{r+1}, \ldots, y_m)\\ \label{eqf2}
&=(2\sqrt{-1})^{\frac{k-1}2}\delta (f_1)({x}_1)\cdots \delta (f_r)(x_r)f_{r+1}(y_{r+1})\cdots f_{m}(y_m), \ \mbox{$k$ odd}\\ \label{eqf3}
f_1&\circledast\cdots\circledast f_{m}(z_1, \ldots, z_m)=0, \qquad \mbox{other type of $z_i$'s}
\end{align}
where $x_i\in\mathscr{B}_{\mu_i}^{(0)}$, $y_i\in\mathscr{B}_{\mu_i}^{(1)}$, $z_i\in\mathscr{B}_{\mu_i}$,
and $[a]$ denotes the maximum integer
$\leq a$. We remark the last formula (\ref{eqf3}) is verified with our updated trivial values stated in
Prop. \ref{tensorprod}.

Now we can compute the associated spin characters on odd conjugacy
classes for $\widetilde{H}_n$. Let $r=[\frac{n}{2}]$, and
$L_n=\mathscr{C}_ne$, where
\begin{equation}
e=e_1e_2\cdots e_r, \ \ \ \
e_i=\frac{1}{2}(1+\sqrt{-1}\xi_{2i-1}\xi_{2i}).
\end{equation}
Then $e_i$ are commutative idempotents, and $L_n$ is a simple super $\mathscr{C}_n$-module.

Note that $\xi^i=\xi_{2i-1}e_i=\frac12(\xi_{2i-1}+\sqrt{-1}\xi_{2i})$ ($i=1, \ldots, r$) generate an exterior
algebra: $\{\xi^i, \xi^j\}=0$ due to the fact that $e_i\xi_{2i-1}e_i=0$.
For each $\varepsilon=(\varepsilon_1,\ldots,\varepsilon_r)\in
\mathbb{Z}^r_2$, define
$\xi^{\varepsilon}=\xi_1^{\varepsilon_1}\xi_3^{\varepsilon_2}\cdots\xi_{2r-1}^{\varepsilon_r}e=
(\xi_1^{\varepsilon_1}e_1)(\xi_3^{\varepsilon_2}e_2)\cdots
(\xi_{2r-1}^{\varepsilon_r}e_r)$. If $n$ is even (resp. odd), then
$\{\xi^{\varepsilon}| \varepsilon\in \mathbb{Z}_2^{r}\}$ (resp.
$\{\xi^{\varepsilon}| \varepsilon\in \mathbb{Z}_2^r\}\cup \{\xi^{\varepsilon}\xi_n| \varepsilon\in \mathbb{Z}_2^r\}$) is a basis
of $L_n$. If $n$ is odd, define $c_n\in End^1_{\mathscr{C}_n}(L_n)$
by
$c_n(\xi^{\varepsilon}\xi_n^{\epsilon})=(-1)^{\sum_i\varepsilon_i+\epsilon}(\xi^{\varepsilon}
\xi_n^{\epsilon+1})$, where
$\varepsilon=(\varepsilon_1,\ldots,\varepsilon_r)\in \mathbb{Z}^r_2$
and $\epsilon\in \mathbb{Z}_2$. Here $c_n^2=-1$.

By our earlier discussion, irreducible $\widetilde{H}_n$-supermodules are indexed by strict partitions $\nu$ and they are type $M$ (resp. $Q$) iff $l(\nu)$ are even (resp. odd). Correspondingly
 double irreducible spin characters are indexed by strict partitions $\nu$ with even $l(\nu)$, and pairs of irreducible associate spin characters are indexed by strict partitions $\nu$ with odd $l(\nu)$.

Let $\xi_{\nu}$ be the character of the irreducible supermodule associated with $\nu$. If $l(\nu)$ is even, then $\xi_{\nu}$ vanishes at split conjugacy classes
of degree one, i.e., $\xi_{\nu}(D_{\rho}^{\pm})=0$ for
any odd strict partition $\rho\in\mathcal{SP}^1$ (as $D_{\rho}^{\pm}$ are the only split odd classes by Theorem
\ref{p:splitconj}). Therefore, $\xi_{\nu}$'s are completely given by
Schur's $Q$-functions in this case (see Prop. \ref{JO HG}). So we are left with the task of fixing $\xi_{\nu}$ with odd $l(\nu)$.
In this case $\xi_{\nu}$ breaks into a pair of associate spin characters $\xi_{\nu}^{\pm}$. Again the character
values on the even split classes are given by Schur's $Q$-functions, and we only need to worry about
$\xi_{\nu}$ on the odd split conjugacy classes $\mathcal{SP}^1$.  The following result computes the remaining character values for
$\mathscr{B}_n$ that correspond to strict partitions of odd lengths.

\begin{theorem}
For $\nu=(\nu_1,\ldots,\nu_l)\in \mathcal{SP}_n$, let
$\xi_{\nu}$ be the corresponding spin supercharacter of $\mathscr{B}_n$.
When $l(\nu)$ is odd, $\xi_{\nu}$ is decomposed into two
associate ordinary spin characters  $\xi^{\pm}_{\nu}$ of
$\widetilde{H}_n$. Moreover, $(\xi_{\nu}^{\pm}){(\mu^{+})}=0$
for $\mu\neq\nu \in \mathcal{SP}^1_n$ and for $\mu=\nu$,
\begin{equation}\xi_{\nu}^{\pm}(\nu^{+})=\pm2^{\frac{l(\nu)}{2}}\cdot (\sqrt{-1})^{\frac{n-m}{2}}\sqrt{\frac{\nu_1\nu_2\cdots\nu_{l(\nu)}}{2}},\end{equation}
where $\mu^{+}\in D^{+}_{\mu}=\theta^{-1}({C}_{\emptyset, \mu})$ and $m$ is the number of odd parts in $\nu$.
\end{theorem}
\begin{proof} We prove the formula by induction on $l(\nu)$.
We first consider the case $\nu=(n)$, and divide into two cases according to the parity of $n$.

 Case (1). If $n$ is odd,
the $\mathscr{C}_n$-module $L_n$ is of type $Q$ and the $\mathscr{A}_n$-character $\zeta_{(n)}$ is of type $M$. This is not a special case
of the formula, but we need to compute the difference
character for the inductive procedure. The element
$\sigma^{((n), \emptyset)}$ of $\mathscr{B}_n$ is given by
$\sigma^{((n),\emptyset)}=\omega_1\omega_2\cdots
\omega_{n-1}\in \mathscr B_n^{(0)}$. Using Eq. (\ref{e:theta}), we obtain that 
\begin{equation}
\vartheta^{-1}(\omega_1\omega_2\cdots \omega_{n-1})=
(\frac{1}{\sqrt{2}}(\xi_1-\xi_2)\otimes \varsigma_1)(\frac{1}{\sqrt{2}}(\xi_2-\xi_3)\otimes \varsigma_2)\cdots,
\end{equation}
 From \cite{Ya1}, we have $\varsigma_i\xi_j=-\xi_j\varsigma_i$, then rearrange the above product into the following form
\begin{equation}
\begin{split}
2^{-\frac{n-1}2}(-1)^{\frac{(n-1)(n-2)}{2}}\prod_{j=1}^{n-1}(\xi_j-\xi_{j+1})\otimes (\varsigma_1\varsigma_2\cdots\varsigma_{n-1}).
\end{split}
\end{equation}
Therefore, the value of $\xi^+_{(n)}=\frac12\Delta(ch(L_n)\hat\otimes\zeta_{(n)})$ at the even element
equals to
\begin{equation}
\begin{split}
&\frac12ch(L_n)\hat{\otimes}\zeta_{(n)}(\vartheta^{-1}(\sigma^{((n),\emptyset)}))\\
=&(-1)^{\frac{(n-1)(n-2)}{2}}2^{-\frac{n+1}2}ch(L_n)(\prod_{j=1}^{n-1}(\xi_j-\xi_{j+1}))
\zeta_{(n)}(\varsigma^{(n)}),
\end{split}
\end{equation}
where $\varsigma^{(n)}=\varsigma_1\varsigma_2\cdots\varsigma_{n-1}$.
As the coefficient of 1 in $\prod_{j=1}^{n-1}(\xi_j-\xi_{j+1})$ is
$(-1)^{(n-1)/2}$, it follows that
\begin{equation}
ch(L_n)(\prod_{j=1}^{n-1}(\xi_j-\xi_{j+1}))
=2^{\frac{n+1}{2}}\cdot (-1)^{\frac{n-1}{2}}.
\end{equation}
On the other hand, $\delta(\zeta_{(n)})(\varsigma^{(n)})=(\sqrt{-1})^{(n-1)/2}\sqrt{n}$ (note $n$ is odd). So we get
the difference of the character is equal to
\begin{equation}
\begin{split}
\delta(\xi^{+}_{(n)})((n))=&\frac{1}{2}ch(L_n)\hat{\otimes}\delta(\zeta_{(n)})(\vartheta^{-1}(\sigma^{((n),\emptyset)})\\
=&(\sqrt{-1})^{\frac{n-1}{2}}\sqrt{n}.
\end{split}
\end{equation}

Case (2). When $n$ is even, $\sigma^{(\emptyset,(n))}=\omega_0\omega_1\cdots \omega_{n-1}\in\mathscr B_n^{(1)}$ and its preimage equals to
\begin{equation}
\begin{split}
&\vartheta^{-1}(\omega_0\omega_1\cdots \omega_{n-1})\\
=&2^{-\frac{n-1}2}\cdot(-1)^{\frac{(n-1)(n-2)}{2}}\xi_1\prod_{j=1}^{n-1}(\xi_j-\xi_{j+1})\otimes (\varsigma_1\varsigma_2\cdots\varsigma_{n-1}),
\end{split}
\end{equation}
which is a product of an even element and an odd element.
Note that $ch(L_n)(\xi_1\prod_{j=1}^{n-1}(\xi_j-\xi_{j+1}))=2^{\frac{n}{2}}\cdot(-1)^{\frac{n-2}{2}}$
and it is of type $M$ due to $n$ being even.
Since $\xi_{(n)}$ is of type $Q$, so $\xi_{(n)}^+=\frac12\Delta(\xi_{(n)})$.
It follows from
 $\zeta^+_{(n)}(n)=(\sqrt{-1})^{\frac{n}{2}}\sqrt{\frac{n}{2}}$ (Prop. \ref{Sch})  that
\begin{equation}
\begin{split}
&\xi_{(n)}^{+}((n))=\frac{1}{2}\Delta(ch(L_n)\hat{\otimes}\zeta^+_{(n)})[\vartheta^{-1}(\sigma^{(\emptyset,(n))})]\\
=&2^{-\frac{n-1}2}\cdot(-1)^{\frac{(n-1)(n-2)}{2}}ch(L_n)(\xi_1\prod_{j=1}^{n-1}(\xi_j-\xi_{j+1}))\cdot \zeta_{(n)}^{+}(\varsigma_{(n)})\\
=&2^{-\frac{n-1}2}\cdot(-1)^{\frac{(n-1)(n-2)}{2}}\cdot(-1)^{\frac{n-2}{2}}2^{\frac{n}{2}}\cdot(\sqrt{-1})^{\frac{n}{2}}\sqrt{\frac{n}{2}}\\
=&(\sqrt{-1})^{\frac{n}{2}}\sqrt{n}.
\end{split}
\end{equation}
Next, suppose $\nu=(\nu_1,\nu_2,\ldots,\nu_l)$ is any strict partition of $n$ with odd $l=l(\nu)$,
say $\nu_1,\ldots,\nu_m$ are odd and $\nu_{m+1},\ldots,\nu_l$ are even, then by
(\ref{eqf}-\ref{eqf2})  it follows that
\begin{equation}
\begin{split}
\xi_{\nu}^{+}(\nu^{+})=&2^{[\frac{l}{2}]}(\sqrt{-1})^{\frac{\nu_1+\cdots+\nu_m-m}2}\cdot\sqrt{\nu_1\cdots\nu_m}\\
&\cdot(\sqrt{-1})^{\frac{\nu_{m+1}+\cdots+\nu_l}2}\cdot\sqrt{\nu_{m+1}\cdots\nu_l}\\
=&2^{\frac{l}{2}}(\sqrt{-1})^{\frac{n-m}{2}}\sqrt{\frac{\nu_1\nu_2\cdots\nu_l}{2}}.
\end{split}
\end{equation}
Note that $\langle\xi_{\nu}^{+},\xi_{\nu}^{+}\rangle_{\mathcal{SP}^1_n}=1/2$. On the other hand
\begin{equation}
\begin{split}
\langle\xi_{\nu}^{+},\xi_{\nu}^{+}\rangle_{\mathcal{SP}_n^1}\geq & \frac{1}{|\widetilde{H}_n|}\sum_{g\in D^{+}_{\emptyset, \nu}\cup D^{-}_{\emptyset, \nu}}|\xi^{+}_{\nu}(g)|^2\\
=&\frac{1}{2^{1+l}z_{\nu}}2\cdot|\xi^{+}_{\nu}{(\nu^{+})}|^2\\
=&\frac{1}{2^{1+l}\nu_1\cdots\nu_{l}} \cdot 2\cdot\left|2^{\frac{l}{2}} \sqrt{\frac{\nu_1\cdots\nu_{l}}{2}}\right|^2=\frac{1}{2}.\\
\end{split}
\end{equation}
Therefore $\xi_{\nu}^{+}(\mu^{+})=0$ for $\mu\neq \nu$.
\end{proof}

 For $\nu\in \mathcal{SP}_n$, we have that
 $\xi_{\nu}^{+}(\a)=2^{\frac{l(\a^{+})-1}{2}}\zeta_{\nu}^{+}(\a^{+})$ for $\a\in \mathcal {OP}_n$ and $\xi_{\nu}^{+}(\nu^{+})=2^{\frac{l(\nu)}{2}}(\sqrt{-1})^{\frac{l(\nu)-m-1}{2}}\zeta_{\nu}^{+}(\nu^{+})$ for $\nu\in \mathcal{SP}_n^1$, where $\a^{+}\in D_{\a}^{+}$ and $m$
 is the number of odd parts in $\nu$.

\section{Spin characters of $\widetilde{H}\Gamma_n$.}

We first recall the parametrization of irreducible spin $\widetilde{H}\Gamma_n$-supermodules.

\begin{proposition} \cite{JW}\label{p:supermod}
The irreducible double spin representations over $\widetilde{H}\Gamma_n$ are indexed by  strict partition
valued functions with even length on $\Gamma^{*}$, and the pairs of irreducible associate spin representations
are indexed by strict partition-valued functions with odd length on $\Gamma^{*}$.
\end{proposition}
A spin class function on $\widetilde{H}\Gamma_n$ is a  function $f:\widetilde{H}\Gamma_n\longrightarrow \mathbb{C}$ such that $f(zx)=-f(x)$
for any element $x$ of a conjugacy class.  Let
$R(\widetilde{H}\Gamma_n)$ be the space of complex-valued spin class functions on $\widetilde{H}\Gamma_n$. The usual bilinear form is defined by
\begin{equation}
\begin{split}
\langle f,g\rangle_{\widetilde{H}\Gamma_n}=&\frac{1}{|\widetilde{H}\Gamma_n|}\sum_{\tilde{x}\in \widetilde{H}\Gamma_n }f(\tilde{x})g(\tilde{x}^{-1})\\
=&\sum_{\rho\in \mathcal{OP}_n(\Gamma_{*})\cup \mathcal{SP}^1_n(\Gamma_{*})}[\frac{1}{|\widetilde{H}\Gamma_n|}\sum_{\widetilde{x}\in D_{\rho}^{+}\cup D_{\rho}^{-}}f(\widetilde{x})\overline{g(\widetilde{x})}]\\
=&\sum_{\rho\in \mathcal{OP}_n(\Gamma_{*})\cup \mathcal{SP}^1_n(\Gamma_{*})}\frac{1}{2^{l(\rho)}Z_{\rho}}f(D_{\rho}^{+})\overline{g(D_{\rho}^{+})},
\end{split}
\end{equation}
where  and $f, g\in R(\widetilde{H}\Gamma_n)$.

A spin super class function $\phi$ on $\widetilde{H}\Gamma_n$ is a spin class function such that it vanishes further on
odd strict conjugacy classes. Let $R^{-}(\widetilde{H}\Gamma_n)$ be the $\mathbb{C}$-span space of spin super class functions on $\widetilde{H}\Gamma_n$. It is easy to see that the spin super-characters of $\widetilde{H}\Gamma_n$ form a $\mathbb{C}$-basis of $R^{-}(\widetilde{H}\Gamma_n)$, therefore the standard bilinear form on $R^{-}(\widetilde{H}\Gamma_n)$
is given by
\begin{equation}
\begin{split}
\langle\phi,\varphi\rangle_{\widetilde{H}\Gamma_n}=\frac{1}{|\widetilde{H}\Gamma_n|}\sum_{\tilde{x}\in \widetilde{H}\Gamma_n }\phi(\tilde{x})\varphi(\tilde{x}^{-1})\\
=\sum_{\rho\in \mathcal{OP}_n(\Gamma_{*})}\frac{1}{2^{l(\rho)}Z_{\rho}}\phi(D_{\rho}^{+})\overline{\varphi(D_{\rho}^{+})},
\end{split}
\end{equation}
where $\phi,\varphi\in R^{-}(\widetilde{H}\Gamma_n)$, as the (super)character values vanish on odd split conjugacy classes.

Wang and the second author \cite{JW} used the vertex operator calculus to compute spin characters of all simple $\widetilde{H}\Gamma_n$-supermodules. In terms of our current discussion
this means that the ordinary irreducible spin character values of $\widetilde{H}\Gamma_n$
at even conjugacy classes are determined by transition matrix
between generalized (or more generally wreath product) Schur $Q$-functions
and power-sum symmetric functions. Moreover, if $\chi$ is an irreducible spin character of
$\widetilde{H}\Gamma_n$ then
\begin{equation}
\begin{split}
\langle\chi,\chi\rangle_{\mathcal{OP}_n(\Gamma_{*})}=1, \ \ \langle\chi,\chi\rangle_{\mathcal{SP}^1_n(\Gamma_{*})}=0 \hbox{~for~} \chi \hbox{~a~double~spin}, \\
\langle\chi,\chi\rangle_{\mathcal{OP}_n(\Gamma_{*})}=\langle\chi,\chi\rangle_{\mathcal{SP}^1_n(\Gamma_{*})}=\frac{1}{2} \ \ \hbox{for~} \chi \hbox{~an~associate~spin}.
\end{split}
\end{equation}

In the following we will compute the values of the associate spin characters
at odd strict conjugacy classes, which are not given by the theory of symmetric functions or
vertex operator calculus (see table 1). Recall that the general element $(g, \pm a_I\sigma)\in \widetilde{H}\Gamma_n$ is even or odd
according to the parity of the cardinality of $I$.

Let $U_{\gamma}$ be the irreducible $\Gamma$-module afforded by the
irreducible character $\gamma\in\Gamma^{*}$. For each strict partition $\nu$ of $n$
let $V_{\nu}$ be the corresponding irreducible spin $\widetilde{H}_n$-supermodule afforded by
 the spin super character $\xi_{\nu}$.
 Now
let $\la=(\la_{\gamma})_{\gamma\in\Gamma^{*}}$ to be a partition-valued function  on $\Gamma^{*}$, we recall
the following result which describes the
 corresponding simple supermodule of $\widetilde{H}\Gamma_{\la}$.

\begin{proposition} \label{p1} \cite{JW}
For each strict partition-valued function $\la=(\la_{\gamma})_{\gamma\in\Gamma^{*}}\in \mathcal{SP}_n(\Gamma^{*})$, the tensor product
$\hat{\bigotimes}_{\gamma\in\Gamma^{*}}(U_{\gamma}^{\otimes |\la_{\gamma}|}\otimes V_{\la_{\gamma}})$
decomposes completely into $2^{[\frac{m}{2}]}$ copies of an irreducible spin $\widetilde{H}\Gamma_{\la}$-supermodule $W_{\la}$,
where $m$ denotes the number of the partitions $\la_{\gamma}$ such that  $l(\la_{\gamma})$ is odd.
Then the induced supermodule $Ind_{\widetilde{H}\Gamma_{\la}}^{\widetilde{H}\Gamma_{n}}W_{\la}$ is the irreducible
spin $\widetilde{H}\Gamma_{n}$-supermodule indexed by $\la$, and it is of type $M$ or $Q$ according to $l(\la)$
is even or odd.
\end{proposition}

Let $\la$ be a partition-valued function $\la=(\la_x)_{x\in X}$, where
$$\la_x=(\la_x(1), \la_x(2), \ldots)$$
is a partition for each color $x\in X$. We denote by $\bar{\la}$ the ordinary partition obtained from $\la$ by forgetting all the colors, i.e.
its parts consist of all $\la_x(i)$, $x\in X$.
For a finite set $X$, let $\la=(\la_{x})_{x\in X}\in \mathcal{SP}_n(X)$ and define
$J_{\la}=\{x\in X|l(\la_{x}) \hbox{~is odd }\}$ and $J_{\la}^{'}=\{x\in X|l(\la_{x})\hbox{~is even }\}$
such that  $J_{\la}\cup J_{\la}^{'}=X$.
We denote $l(J_{\la})=\sum_{x\in J_{\la}}l(\la_x)$, $||\la||_{J_{\la}}=\sum_{x\in J_{\la}}|\la_x|$
and $m_{J_{\la}}=\sum_{x\in J_{\la}}m_{\la_x}$, where $m_{\la_x}$ is the number of odd parts in $\la_x$.
For each strict partition $\la_{\gamma}$ let $\xi_{\la_{\gamma}}$ be the corresponding irreducible
spin super-character of $\widetilde{H}_{|\la_{\gamma}|}$.

It follows from Proposition \ref{p1} that the irreducible super-character
$$ch(W_{\la})=2^{-[\frac{|J_{\la}|}{2}]}\prod_{\gamma\in\Gamma^{*}}(\gamma^{\otimes |\la_{\gamma}|}\otimes \xi_{\la_{\gamma}}).$$
The induced character $\chi_{\la}=\mbox{Ind}_{\widetilde{H}\Gamma_{\la}}^{\widetilde{H}\Gamma_n}
 ch(W_{\la}) $
is a double spin character when $l(\la)$ is even and $\chi_{\la}^{\pm}=\mbox{Ind}_{\widetilde{H}\Gamma_{\la}}^{\widetilde{H}\Gamma_n}
 ch(W_{\la}^{\pm})$ are associate spin characters when $l(\la)$ is odd.

For a partition-valued function $(\la_{\gamma})=(\la_1, \ldots, \la_l)_{\gamma}$, the partition $\overline{(\la_{\gamma})}=(\la_1, \ldots, \la_l)$ gives rise to
a collection $[\gamma]$ of partition-valued functions $\rho=(\rho_{c})_{c\in\Gamma_{*}}=(\rho_{c_1},\ldots,\rho_{c_{|\Gamma_{*}|}})$ on $\Gamma_*$ by $\overline{\rho}=(\overline{\rho}_{\gamma_1}, \ldots, \overline{\rho}_{\gamma_{|\Gamma^*|}})$ such that
$\overline{\rho}=\overline{(\la_{\gamma})}$.  Clearly $|[\gamma]|=|\Gamma_*|^{l(\rho)}$. For brevity a partition-valued function inside $[\gamma]$ is
usually written as
$\rho_{\gamma}$.  Now suppose $\rho=(\rho_c)_{c\in\Gamma_{*}}\in \mathcal{SP}^1_n(\Gamma_{*})$ is an odd strict partition-valued function and  $\la_{\gamma}\in \mathcal{SP}_{|\la_{\gamma}|}(\Gamma^{*})$ with $l(\la)$ being odd.
Then using (\ref{eqf2}) we have
\begin{equation}\label{eq,2}
\begin{split}ch(W^{+}_{\la})(D_{\rho}^{+})
=&(2\sqrt{-1})^{[\frac{|J_{\la}|}{2}]} \prod_{\gamma\in J_{\la}}
\gamma^{\otimes |\la_{\gamma}|}\otimes\xi^{+}_{\la_{\gamma}}(D_{\rho_{\gamma}}^{+})\cdot\\
& \prod_{\gamma\in J^{'}_{\la}}
\delta(\gamma^{\otimes |\la_{\gamma}|}\otimes\xi_{\la_{\gamma}})(D_{\rho_{\gamma}}^{+}).
\end{split}
\end{equation}

By Prop. \ref{p:supermod}, irreducible $\widetilde{H}\Gamma_n$-supermodules are indexed by strict parti\-tion-valued functions $\lambda\in
\mathcal{SP}(\Gamma^*)$ and they are type $M$ (resp. $Q$) iff $l(\lambda)$ are even (resp. odd). Correspondingly
 double irreducible spin characters are indexed by strict partition values functions $\lambda$ with even $l(\lambda)$, and pairs of irreducible associate spin characters are indexed by strict partitions $\lambda\in\mathcal{SP}(\Gamma^*)$ with odd $l(\lambda)$.

Let $\xi_{\lambda}$ be the character of the irreducible supermodule associated with partition-valued function
$\lambda$. If $l(\nu)$ is even, then $\xi_{\lambda}$ vanishes at split conjugacy classes
of degree one, i.e., $\xi_{\lambda}(D_{\rho}^{\pm})=0$ for
any odd strict partition-valued function $\rho\in\mathcal{SP}^1(\Gamma_*)$ (see Proposition \ref{p:splitconj}). Therefore, $\xi_{\lambda}$'s are completely given by
the vertex operator calculus in this case (see \cite{JW}). So we are left with the task of fixing $\xi_{\lambda}$ with odd $l(\lambda)$.
In this case $\xi_{\lambda}$ breaks into a pair of associate spin characters $\xi_{\lambda}^{\pm}$. Again the character
values on the even split classes are given by vertex operator calculus, and we only need to worry about
$\xi_{\lambda}$ on the odd split conjugacy classes $D_{\rho}, \rho\in\mathcal{SP}^1(\Gamma_*)$.  The following result computes the remaining character values for
$\widetilde{H}\Gamma_n$ that correspond to strict partition-valued functions of odd lengths.


The following theorem gives the irreducible associate spin character values at conjugacy classes
indexed by odd strict partition-valued functions for $\widetilde{H}\Gamma_n$.

\begin{theorem}\label{Segv}
Let $\chi^{+}_{\la}$ be an associate spin character (i.e. $\la=(\la_{\gamma})_{\gamma\in\Gamma^{*}}$ $\in \mathcal{SP}_n(\Gamma^{*})$, $l(\la)$ odd),
 for $\rho\in\mathcal{SP}^1_n(\Gamma_{*})$ we have

(i) when $\rho=(\rho_{\gamma})_{\gamma\in\Gamma^*}$ such that $\rho_{\gamma}\in\mathcal{OSP}_{|\rho_{\gamma}|}(\Gamma_{*})$ for $\gamma\in J_{\la}^{'}$ and $\rho_{\gamma}\in [\la_{\gamma}]$
 for $\gamma\in J_{\la}$ then
\begin{equation}\nonumber
\begin{split}\chi^{+}_{\la}(D^{\pm}_{\rho})
=&\pm
\prod_{\gamma\in\Gamma^{*}}\big(\prod_{c\in\Gamma_{*}}\gamma(c)^{l(\rho_{\gamma}(c))}\big)\cdot\prod_{\gamma\in
J_{\la}^{'}} \xi_{\la_{\gamma}}(t_{\rho_{\gamma}})\cdot\\
& 2^{\frac{l(J_{\la})}{2}}(\sqrt{-1})^{\frac{||\rho||_{J_{\la}}-m_{J_{\la}}}{2}}
\sqrt{\frac{\prod_{\gamma\in
J_{\la}}\prod_{c\in\Gamma_{*}}z_{\rho_{\gamma}(c)}}{2}},
\end{split}
\end{equation}
where $t_{\rho_{\gamma}}\in D_{\rho_{\gamma}}^{+}$ and the value $\prod_{\gamma\in
J_{\la}^{'}}\xi_{\la_{\gamma}}(t_{\rho_{\gamma}})$ is given by
Schur Q-functions.

(ii) $\chi^{+}_{\la}(D_{\rho}^{\pm})=0$, otherwise.
\\
\end{theorem}
\begin{proof} (i) We have already seen that
$$\delta(\gamma^{\otimes |\la_{\gamma}|}\otimes\xi_{\la_{\gamma}})(D^{\pm}_{\rho_{\gamma}})
=\pm\prod_{c\in\Gamma_{*}}\gamma(c)^{l(\rho_{\gamma}(c))}\cdot \delta(\xi_{\la_{\gamma}})(t_{\rho_{\gamma}});$$
$$ \gamma^{\otimes|\la_{\gamma}|}\otimes\xi^{+}_{\la_{\gamma}}(D^{\pm}_{\rho_{\gamma}})
=\pm\prod_{c\in\Gamma_{*}}\gamma(c)^{l(\rho_{\gamma}(c))}\cdot\xi^{+}_{\la_{\gamma}}(t_{\rho_{\gamma}}).$$
Note that we will show that there is only one left coset $T$ of $\widetilde{H}\Gamma_{\la}$ in $\widetilde{H}\Gamma_n$ such that
$(g, a_Is)T=T$. Suppose there are $K_{\rho}$ such
cosets. It follows from Prop. \ref{p1} that (as $l(\la)$ is odd, so $|J_{\la}|$ is odd)
\begin{equation}
\begin{split}\chi^{+}_{\la}(D^{\pm}_{\rho})
=&\pm K_{\rho}\cdot2^{\frac{|J_{\la}|-1}{2}}
\prod_{\gamma\in\Gamma^{*}}\big(\prod_{c\in\Gamma_{*}}\gamma(c)^{l(\rho_{\gamma}(c))}\big)\cdot\prod_{\gamma\in
J_{\la}^{'}} \xi_{\la_{\gamma}}(t_{\rho_{\gamma}})\cdot\\
& \prod_{\gamma\in
J_{\la}}\big(2^{\frac{l(\la_{\gamma})}{2}}(\sqrt{-1})^{\frac{|\rho_{\gamma}|-m_{\la}}{2}}
\sqrt{\frac{\prod_{c\in\Gamma_{*}}z_{\rho_{\gamma}(c)}}{2}}\big)\\
=&\pm K_{\rho}
\prod_{\gamma\in\Gamma^{*}}\big(\prod_{c\in\Gamma_{*}}\gamma(c)^{l(\rho_{\gamma}(c))}\big)\cdot\prod_{\gamma\in
J_{\la}^{'}} \xi_{\la_{\gamma}}(t_{\rho_{\gamma}})\cdot\\
& 2^{\frac{l(J_{\la})}{2}}(\sqrt{-1})^{\frac{||\rho||_{J_{\la}}-m_{J_{\la}}}{2}}
\sqrt{\frac{\prod_{\gamma\in
J_{\la}}\prod_{c\in\Gamma_{*}}z_{\rho_{\gamma}(c)}}{2}}\\
\end{split}
\end{equation}

(ii) The first case:  if $\rho$ can not be decomposed into $\cup_{\gamma\in\Gamma^{*}}(\rho_{\gamma})$
 such that $\rho_{\gamma}\in \mathcal{SP}_{|\la_{\gamma}|}(\Gamma_{*})$,
 then we have $\chi^{+}_{\la}(D^{+}_{\rho})=0$.
 So we can assume that $\rho=\cup_{\gamma\in\Gamma^{*}}(\rho_{\gamma})$ such that $\rho_{\gamma}\in \mathcal{SP}_{|\la_{\gamma}|}(\Gamma_{*})$
and $\chi^{+}_{\la}(D^{+}_{\rho})\neq 0$.

(1) When $\rho_{\gamma}\in\mathcal{OSP}_{|\rho_{\gamma|}}(\Gamma_{*})$ for $\gamma\in J_{\la}^{'}$, we claim that $\rho_{\gamma}\in [\la_{\gamma}]$ for $\gamma\in J_{\la}$. In fact, let $S$ be the set of
conjugacy classes $\rho$ such that $\rho=\cup_{\gamma\in\Gamma^{*}}(\rho_{\gamma})$. It is easy to see that if $\xi_{\la_{\gamma}}(t_{\rho_{\gamma}})$ is nonzero then $\rho_{\gamma}$
must be in
$\mathcal{OSP}_{|\la_{\gamma}|}(\Gamma_{*}):=\mathcal{OP}_{|\la_{\gamma}|}(\Gamma_{*})\cap \mathcal{SP}_{|\la_{\gamma}|}(\Gamma_{*})$ for $\gamma\in J_{\la}^{'}$
and  $\rho_{\gamma}$ must be in $[\la_{\gamma}]$ for $\gamma\in J_{\la}$. Then
\begin{equation}\label{eq:4}
\begin{split}
&\langle{\chi^{+}_{\la}}, \chi^{+}_{\la}\rangle_{\mathcal{SP}^1_n(\Gamma_{*})}\geq \sum_{\rho\in S}\frac{1}{\widetilde{Z}_{\rho}}|\chi^{+}_{\la}(\rho)|^2
=\sum_{\rho\in S}\frac{1}{Z_{\rho}}|\chi^{+}_{\la}(D^{+}_{\rho})|^2\\
=&\sum_{\rho\in S} K_{\rho}^2\frac{1}{\prod_{\gamma\in {\Gamma^{*}}}2^{l(\rho_{\gamma})}\prod_{c\in\Gamma_{*}}z_{\rho_{\gamma}(c)}\zeta_{c}^{l(\rho_{\gamma}(c))}}\cdot\prod_{\gamma\in J_{\la}}\prod_{c\in\Gamma_{*}}\gamma(c)^{2l(\rho_{\gamma}(c))}\cdot\\
 & \prod_{\gamma\in
J_{\la}^{'}}\prod_{c\in\Gamma_{*}} \gamma(c)^{2l(\rho_{\gamma}(c))}\xi^2_{\la_{\gamma}}(t_{\rho_{\gamma}})\cdot2^{l(J_{\la})} \frac{\prod_{\gamma\in J_{\la}^{'}}\prod_{c\in\Gamma_{*}}z_{\rho_{\gamma}(c)}}{2}\\
\geq&\frac{1}{2}\sum_{\rho\in S}\prod_{\gamma\in J_{\la}}\prod_{c\in\Gamma_{*}}\frac{\gamma(c)^{2l(\rho_{\gamma}(c))}}{\zeta_{c}^{l(\rho_{\gamma}(c))}}\cdot \prod_{\gamma\in J_{\la}^{'}}\frac{\prod_{c\in\Gamma_{*}}\gamma(c)^{2l(\rho_{\gamma}(c))}\xi^2_{\la_{\gamma}}(t_{\rho_{\gamma}})}{2^{l(\rho_{\gamma})}\prod_{c\in\Gamma_{*}}
\zeta_c^{l(\rho_{\gamma}(c))}}\\
\geq&\frac{1}{2}\prod_{\gamma\in J_{\la}}(\sum_{\rho_{\gamma}\in [\la_{\gamma}]}\frac{1}{\prod_{c\in\Gamma_{*}}\zeta_{c}^{l(\rho_{\gamma}(c))}}\prod_{c\in\Gamma_{*}}\gamma(c)^{2l(\rho_{\gamma}(c))})\cdot\\
&\prod_{\gamma\in {J_{\la}^{'}}}(\sum_{\rho_{\gamma}\in \mathcal{OSP}_{|\la_{\gamma}|}(\Gamma_{*}) }\frac{\prod_{c\in\Gamma_{*}}\gamma(c)^{2l(\rho_{\gamma}(c))}\xi^2_{\la_{\gamma}}(t_{\rho_{\gamma}})}{2^{l(\rho_{\gamma})}Z_{\rho_{\gamma}}})\cdot \\
 =&\frac{1}{2}\prod_{\gamma\in J_{\la}}\langle\gamma^{\otimes l(\rho_{\gamma})},\gamma^{\otimes l(\rho_{\gamma})}\rangle_{\Gamma^{l(\rho_{\gamma})}}
 =\frac{1}{2}.
\end{split}
\end{equation}
We have known that $\langle{\chi^{+}_{\la}}, \chi^{+}_{\la}\rangle_{\mathcal{SP}^1_n(\Gamma_{*})}=\frac{1}{2}$, so $\chi^{+}_{\la}(D^{\pm}_{\rho})=0$ if  $\rho_{\gamma}\notin [\la_{\gamma}]$ for $\gamma\in J_{\la}$
and also $K_{\rho}=1$.

(2) If $\rho_{\gamma}\notin\mathcal{OSP}_{|\la_{\gamma}|}(\Gamma_{*})$ for $\gamma\in J_{\la}^{'}$, then there exists  one $\rho_{\gamma}$ not in $\mathcal{OP}_{|\la_{\gamma}|}(\Gamma_{*})$ for $\gamma\in J_{\la}^{'}$. Meanwhile, since $\xi_{\la_{\gamma}}$ is a double spin character
 when $\gamma\in J_{\la}^{'}$, it is known that it only has nonzero values at even conjugacy classes (i.e. in $\mathcal{OP}_n(\Gamma_{*})$). Hence we have
$\xi_{\la_{\gamma}}(t_{\rho_{\gamma}})=0$, therefore ${\chi^{+}_{\la}}(D^{\pm}_{\rho})=0$.
\end{proof}

With this result we have determined all remaining character values of the spin characters
for $\widetilde{H}\Gamma_n$.

 \bibliographystyle{amsalpha}

 \bibliographystyle{amsalpha}
\end{document}